\documentclass[leqno,11pt]{amsart}

\usepackage{bm}
\usepackage{amssymb,amsfonts}
\usepackage{amsmath,latexsym}
\usepackage{pdfsync}
\usepackage{xcolor}
\usepackage{mathrsfs}

\newtheorem{thm}{Theorem}[section]
\newtheorem{cor}[thm]{Corollary}
\newtheorem{lem}[thm]{Lemma}

\theoremstyle{definition}

\newtheorem{rem}[thm]{Remark}
\numberwithin{equation}{section}
\def\eqdefa{\buildrel\hbox{\footnotesize def}\over =}
\newcommand{\R}{\mathbb R}
\newcommand\E{\mathcal{E}}
\newcommand\J{\mathscr J}
\newcommand\B{\mathscr{B}}
\newcommand\N{\mathbb{N}}

\newcommand{\dx}{\,{\rm d}x}

\allowdisplaybreaks[4]
\begin{document}
\thispagestyle{empty}

\vspace{1 true cm} {
\title[Regularity for 3D Navier--Stokes equations]
 {Prodi--Serrin condition for 3D Navier--Stokes equations via one directional derivative of velocity}%
\author[H.Chen]{Hui Chen}%
\address[H. Chen]
 {School of Science, Zhejiang University of Science and Technology, Hangzhou, 310023, People's Republic of China }
\email{chenhui@zust.edu.cn}
\author[W. Le]{Wenjun Le}%
\address[W. Le]
 {School of Science, Zhejiang University of Science and Technology, Hangzhou, 310023, People's Republic of China}
\email{le$_-$wenjun@163.com}
\author[C. Qian]{Chenyin Qian}
\address[C. Qian]%
{Department of Mathematics, Zhejiang  Normal University Jinhua,
            321004, China}
\email{qcyjcsx@163.com }

%

\begin{abstract}
In this paper, we consider the conditional regularity of weak solution to the 3D Navier--Stokes equations. More precisely, we prove that if  one directional derivative of velocity, say $\partial_3\bm u,$ satisfies
  $\partial_3\bm u \in L^{p_0,1}(0,T; L^{q_0}(\R^3))$ with  $\frac{2}{p_{0}}+\frac{3}{q_{0}}=2$ and $\frac{3}{2}<q_0< +\infty,$ then the  weak solution  is regular on $(0,T].$ The proof is based on the new local energy estimates introduced by Chae-Wolf (arXiv:1911.02699) and Wang-Wu-Zhang (arXiv:2005.11906).
\end{abstract}

\maketitle


\noindent {{\sl Key words:} Navier--Stokes equations; Regularity of weak solutions; Serrin-Prodi condition}

\vskip 0.2cm

\noindent {\sl AMS Subject Classification (2000):} 35Q35; 35Q30; 76D03

\section{Introduction}
We consider the Cauchy problem for incompressible Navier--Stokes equations in $\R^3\times(0,\infty)$.
\begin{equation}\label{NS}
\left\{
\begin{array}{ll}
\vspace{3pt}
&\partial_{t}\bm{u}+(\bm{u}\cdot\nabla) \bm{u}-\Delta  \bm{u}+\nabla \pi=0,\\
\vspace{3pt}
&\nabla\cdot  \bm{u}=0~,\\
&u|_{t=0}=\bm{u}_{0}~,
\end{array}
\right.
\end{equation}
where  $\bm{u}=(u_{1},u_{2},u_{3})$ and $ \pi$ stand for  the
velocity field and a scalar pressure of the viscous incompressible fluid, respectively.

For every $\bm{u}_0\in L^{2}(\R^{3})$ with $\mathrm{div}~\bm{u}_{0}=0$ in the sense of distribution, a global weak solution $\bm{u}$ to the Navier--Stokes equations \eqref{NS}, which  satisfies the energy inequality
\begin{equation}\label{EN}
\|\bm{u}(\cdot,t)\|_{L^{2}(\R^3)}^{2}+2\int_{0}^{t}\|\nabla
\bm{u}(\cdot,s)\|_{L^{2}(\R^3)}^{2}\text{d}s  \leq\|\bm{u}_{0}\|_{L^{2}(\R^3)}^{2},\quad \text{ for all } \ t>0,
\end{equation}
was constructed by Leray  \cite{Leray1934} and Hopf \cite{Hopf1951}.  However, the uniqueness and regularity of such weak solution is still one of the most challenging open problems in the field of mathematical fluid mechanics.

One essential work is usually referred as Prodi--Serrin (P--S) conditions (see \cite{Escauriaza2003,Prodi1959,Serrin1962,Takahashi1990} and the references therein.), i.e. if the weak solution $\bm{u}$  satisfies
\begin{equation} \bm u\in L^{p}(0,T; L^{q}(\R^{3})),\ \
\frac{2}{p}+\frac{3}{q}=1,\ 3\leq q\leq \infty,
\end{equation} then the weak solution is regular in $(0,T]$. Along with the pioneering works of Prodi and Serrin,
Beir$\tilde{\mbox{a}}$o da Veiga \cite{BeiraoDaVeiga1995} established  regularity criteria on the gradient of the velocity field, i.e.
$$
\nabla \bm u\in L^{p}(0,T; L^{q}(\R^{3})),\ \
\frac{2}{p}+\frac{3}{q}=2,\ \frac{3}{2}\leq q\leq \infty.
$$
Later on, many efforts have been made to weakening the above criteria by imposing constraints only on partial components or directional derivatives of velocity field.

There are several notable results \cite{Chemin2016,Chemin2017,Han2019} based on one component of the velocity. For instance, B. Han etc. \cite{Han2019} proved that if $u_3 \in L^{p}(0,T;
\dot{H}^{1/2+2/p}(\R^{3}))$ with $2\leq p<+\infty$, the solution $\bm{u}$ is regular in $(0,T]$.
Very recently, D. Chae and J. Wolf \cite{Chae2019} made an important progress and obtained the regularity of solution to \eqref{NS} under the condition
\begin{align}\label{1.8}
u_{3} \in L^{p}\left(0, T ; L^{q}\left(\R^{3}\right)\right), \quad \frac{2}{p}+\frac{3}{q}<1, \quad 3<q\leq\infty.
\end{align}
W. Wang, D. Wu and Z. Zhang \cite{Wang2020} improved to
\begin{align}\label{1.8a}
u_{3} \in L^{p,1}\left(0, T; L^{q}\left(\R^{3}\right)\right), \quad \frac{2}{p}+\frac{3}{q}=1, \quad 3<q<\infty.
\end{align}
Throughout this paper, $L^{p,1}$ denotes the Lorentz space with respect to the time variable.

For the regularity criteria only involving one directional derivative of velocity , I. Kukavica and M. Zaine \cite{Kukavica2007} get the result
$$
\partial_3\bm u\in L^{p}(0,T; L^q(\R^3)), \frac{2}{p}+\frac{3}{q}=2,\quad \frac{9}{4}\leq q\leq3.
$$
There are also many efforts to extend the range of $q$ for $\partial_3\bm u$, such as \cite{Cao2010, Namlyeyeva2020,Kukavica2007,Zhang2018}. In particular, the first author of this paper and D. Fang and T. Zhang \cite{Chen2020} proved that $\bm{u}$ is regular in $(0,T]$, if
\begin{align*}
\partial_{3} \bm{u} \in L^{p}(0,T;L^{q}(\R^3)),~\frac{2}{p}+\frac{3}{q}=2,~\frac{3}{2} < q\leq 6.
\end{align*}

Along this line of research, we obtain  scaling invariant Prodi-Serrin criteria for $\partial_3 u$  with  optimal range $\frac{3}{2} < q< \infty$. More precisely, we prove the following theorem:
\begin{thm}\label{thm1}
{\sl Let  $\bm u_{0}\in  L^2(\R^3)\cap L^3(\R^3)$  and $(\bm u,\pi)$ be a Leray-Hopf weak solution to the 3D Navier--Stokes equations
\eqref{NS}. If $\bm{u}$ satisfies
	\begin{align}\label{1.9}
	\partial_{3} \bm{u} \in L^{p_{0},1}(0,T;L^{q_{0}}(\R^3)),~\frac{2}{p_{0}}+\frac{3}{q_{0}}=2,~\frac{3}{2} < q_{0}<+ \infty,
	\end{align}
	or
	\begin{align}
	\partial_{3} \bm{u} \in L^{p_{0},1}(0,T;L^{\infty}(\R^3)),~p_{0}>1,
	\end{align}
then $\bm u$ is regular in $\R^3\times(0,T]$.}
\end{thm}
\begin{rem}
Theorem \ref{thm1} is a direct consequence of Theorem  \ref{thm2}  and Remark \ref{rem2} via a similar compactness argument in \cite{Chae2019}. Moreover, the initial data $\bm u_{0}\in  L^2(\R^3)\cap L^3(\R^3)$ implies the local-in-time regularity of weak solution, thus the weak solution is actually suitable weak solution.
\end{rem}

A particular class of weak solution to \eqref{NS} called \emph{suitable weak solution} is introduced by Caffarelli, Kohn  and Nirenberg \cite{Caffarelli1982}. We say that  $(\bm{u}, \pi)$ is a suitable weak solution of \eqref{NS} in open domain $\Omega_{T}=\Omega \times(-T, 0)$, if
\begin{enumerate}
	\item[$(1)$] $\bm{u} \in L^{\infty}\left(-T, 0 ; L^{2}(\Omega)\right) \cap L^{2}\left(-T, 0 ; H^{1}(\Omega)\right)$ and $\pi \in L^{\frac{3}{2}}\left(\Omega_{T}\right)$;
	\item[$(2)$] \eqref{NS} is satisfied in the sense of distribution;
	\item[$(3)$] the local energy inequality holds: for any nonnegative test function $\varphi \in C_{c}^{\infty}\left(\Omega_{T}\right)$  and  $ t \in (-T, 0)$,
	\begin{align}\label{local energy}
	&\int_{\Omega}|\bm{u}(x, t)|^{2} \varphi ~\text{d}x+2 \int_{-T}^{t} \int_{\Omega}|\nabla \bm{u}|^{2} \varphi ~\text{d}x\text{d}s \notag\\
	\leq& \int_{-T}^{t} \int_{\Omega}|\bm{u}|^{2}\left(\partial_{s} \varphi+\Delta \varphi\right)+\bm{u} \cdot \nabla \varphi\left(|\bm{u}|^{2}+2 \pi\right) ~\textrm{d}x\textrm{d}s.
	\end{align}
\end{enumerate}

The important progress in \cite{Caffarelli1982} is that one-dimensional Hausdorff
measure of the possible space-time singular points set for the suitable weak solution $(\bm{u},\pi)$ is zero.
A simple proof is also given by F. Lin in \cite{Lin1998}.
\begin{thm}\label{thm2}
{\sl Let $(\bm{u}, \pi)$ be a suitable weak solution of \eqref{NS} in $\R^{3} \times(-1,0) .$ If $\bm{u}$ satisfies
\begin{equation}\label{R1}
	\partial_{3} \bm{u} \in L^{p_{0},1}\left(-1,0;L^{q_{0}}(B(2))\right),~\frac{2}{p_{0}}+\frac{3}{q_{0}}=2,~\frac{3}{2} < q_{0}<+ \infty,
\end{equation}
then there exists $ n_{0}\geq 1$, such that
\begin{align*}
r^{-2}\|\bm{u}\|_{L^{3}\left(B(r) \times\left(-r^{2}, 0\right)\right)}^{3} \leq C
\end{align*}
for any $0<r\leq2^{-n_{0}}$. Here $B(r)$ is the ball in $\R^3$ with center at the origin and radius $r$.}
\end{thm}
\begin{rem}\label{rem2}
Our method fails in the case $q_{0}=+\infty$. However, replacing \eqref{R1} with the subcritical regularity criteria
	\begin{equation}\label{subc}
	\partial_{3} \bm{u} \in L^{p_{0}}\left(-1,0;L^{\infty}(B(2))\right),p_{0}>1,
	\end{equation}
Theorem \ref{thm2} still holds true.
Actually, we can pick $1<p_{1}<p_{0}$ and $\frac{3}{2}<q_{1}<+\infty$ such that $\frac{2}{p_{1}}+\frac{3}{q_{1}}=2$. Therefore, we can prove it directly by the embedding inequality
\begin{equation*}
\left\|\partial_{3}\bm{u}\right\|_{L^{p_{1},1}\left(-1,0;L^{q_{1}}(B(2))\right)}\leq C\left\|\partial_{3}\bm{u}\right\|_{L^{p_{0}}\left(-1,0;L^{\infty}(B(2))\right)}.
\end{equation*}
\end{rem}

The proof of Theorem \ref{thm2} follows from the  ideal of local energy estimates introduced in \cite{Caffarelli1982,Chae2019,Wang2020}. The non-trivial part is to establish the necessary \textit{a priori} estimates  of the quantity $J=\int_{-1}^{t} \int_{U_{0}(R)} \pi_{0} u_{3} \cdot \partial_{3}\left(\Phi_{n} \eta \psi\right)\dx  \text{d}s$   involving the "non-harmonic" part of the pressure $\pi_{0}$. The difficulty will be overcome by introducing
the mean value function $\overline{\left(u_{3}\right)}_{k}$ defined in \eqref{overline}. More precisely, we adopt a new decomposition of the quantity $J$ in \eqref{J} and \eqref{J_{1}'}. We should point out that the energy inequality in Lemma \ref{energy}  and  Poincar\'{e}'s inequality in Lemma \ref{A3} are crucial to deduce the estimates.

Our paper is organized as follows: in Section 2, we recall some notations and preliminary results ; we establish the \textit{a priori} estimates involving the convection term in Section 3.1 and  the pressure term in Section 3.2; finally, we will complete the proof in Section \ref{section3.3}.

\section{Notations and Preliminary}
In this preparation section, we recall some usual notations and  preliminary results.

For two comparable quantities, the inequality $X\lesssim Y$ stands for $X\leq C Y$ for some positive constant $C$. The dependence of the constant $C$ on other parameters or constants are usually clear from the context, and we will often suppress this dependence.

 \par
 We  also shall use the same  notation as that in  Chae-Wolf \cite{Chae2019}. For $x=(x_{1},x_{2},x_{3}) \in \R^{3}$, we denote $x^{\prime}=(x_{1},x_{2})$  the horizontal variable.

 For $0<R<+\infty$, we denote $B(R)$ the ball in $\R^3$ with center at origin and radius $R$, and $B^{\prime}(R)\subset \R^2$ the ball in the horizontal plane  with center at origin and radius $R$.

We set the spatial cylinder
\begin{equation*}
  U_n(R) \eqdefa B'(R)\times(-r_n, r_n),
\end{equation*}
the  parabolic cylinder
\begin{equation*}
  Q_n(R) \eqdefa U_n(R)\times(-r_n^2, 0),
\end{equation*}
and
\begin{equation*}
  A_n(R) \eqdefa Q_n(R) \backslash Q_{n+1}(R),
\end{equation*}
where $r_{n}=2^{-n},\ n\in\N$.

We take $\Phi_{n}$ the fundamental solution of the backward heat equation
\begin{equation*}
\partial_{t}\Phi_{n}+\partial_{3}^2 \Phi_{n}=0,
\end{equation*}
with singularity at $(0,r_{n}^2)$. More explicitly, we consider $\Phi_n$ given by
$$
\Phi_n(x,t)\eqdefa \frac{1}{\sqrt{4\pi(-t+r_n^2)}}e^{-\frac{x_3^2}{4(-t+r_n^2)}},\quad (x,t)\in\R^3\times(-\infty, 0).
$$

There exist absolute  constants $c_1, c_2>0$ such that for  $j=1,\ldots,n-1$, it holds
\begin{equation}\label{Phi_{n}}
\begin{split}
& \Phi_n\leq c_2 r_j^{-1},\qquad \qquad  |\partial_3\Phi_n|\leq c_2 r_j^{-2},\qquad\quad\text{in}\quad A_j(R),\\
c_1 r_n^{-1}&\leq \Phi_n\leq c_2 r_n^{-1},\quad c_1 r_n^{-2}\leq |\partial_3\Phi_n|\leq c_2 r_n^{-2},\quad\text{in}\quad Q_n(R).
\end{split}
\end{equation}
We denote the energy
\begin{align*}
  E_n(R) &\eqdefa \sup_{t\in (-r_n^2, 0)}\int_{U_n(R)}|\bm u(t)|^2\text{d}x+\int_{-r^2_n}^{0}\int_{U_n(R)}|\nabla \bm u|^2 \dx \text{d}s,\\
  \E\eqdefa & \sup_{t\in (-1, 0)}\int_{\R^3}|\bm u(t)|^2
  \dx+\int_{-1}^{0}\int_{\R^3}|\nabla \bm u|^2 \dx \text{d}s.
\end{align*}
 Next, we denote $L^{q_{v}}_{v}L^{q_{h}}_{h}\left(B^{\prime}(R)\times(-r,r)\right)$ the anisotropic Lebesgue space equipped with the norm
\begin{equation}\label{anisotropic}
\|h(x)\|_{L^{q_{v}}_{v}L^{q_{h}}_{h}\left(B^{\prime}(R)\times(-r,r)\right)}
=\left(\int_{-r}^{r}\|h(\cdot,x_{3})
\|_{L^{q_{h}}\left(B^{\prime}(R)\right)}^{q_{v}}\dx  _{3}\right)^{\frac{1}{q_{v}}}.
\end{equation}
The following lemma ensures the anisotropic Lebesgue space obeys the energy estimates in \cite[Lemma 3.1]{Chae2019}.
\begin{lem}\label{energy}
{\sl	Let $ R\geq \frac{1}{2}$. For $\forall~ 2\leq p \leq \infty$, $2\leq q_{h},q_{v}<+\infty$, $\frac{2}{p}+\frac{2}{q_{h}}+\frac{1}{q_{v}}=\frac{3}{2}$, we have
	\begin{equation}
	\|\bm{u}\|_{L^{p}\left(-r_{n}^{2}, 0 ;L^{q_{v}}_{v}L^{q_{h}}_{h}\left(U_{n}(R)\right)\right)}^{2} \leq C E_{n}(R).
	\end{equation}}
\end{lem}
\begin{proof}
For $0<r\leq 1$, set $U=B^{\prime}(R)\times(-r,r)$, $U_{1}=B^{\prime}(R)\times(-2r,2r)$, $U_{2}=B^{\prime}(R)\times(-4R,4R)$ and $U_{3}=B^{\prime}(1)\times(-4,4)$.

We claim  that
\begin{align}\label{2.4}
\|\bm{u}\|_{L^{q_{v}}_{v}L^{q_{h}}_{h}\left(U\right)} \leq C r^{-\frac{2}{p}}\|\bm{u}\|_{L^{2}(U)}+C\|\bm{u}\|_{L^{2}(U)}^{1-\frac{2}{p}}\|\nabla \bm{u}\|_{L^{2}(U)}^{\frac{2}{p}}.
\end{align}
In fact, given $\bm{u} \in W^{1,2}(U)$, we define the extension
\begin{align*}
\bm{v}_{1}(x):=\left\{\begin{array}{lll}
\bm{u}(x^{\prime},x_{3}) & \text { if } & x_{3} \in (-r,r), \\
\bm{u}\left(x^{\prime}, 2 r-x_{3}\right) & \text { if } &  x_{3} \in[r, 2 r), \\
\bm{u}\left(x^{\prime},-2 r-x_{3}\right) & \text { if } & x_{3} \in(-2 r,-r].
\end{array}\right.
\end{align*}
Then $\bm{v}_{1} \in W^{1,2}(U_{1})$ and it holds
\begin{equation}\label{2.5-1}
\|\nabla \bm{v}_{1}\|_{L^{2}(U_{1})} \leq 3\|\nabla \bm{u}\|_{L^{2}(U)}, \quad\|\bm{v}_{1}\|_{L^{2}(U_{1})} \leq 3\|\bm{u}\|_{L^{2}(U)}.
\end{equation}
Let $\zeta \in C_{c}^{\infty}(-2 r, 2 r)$ denote a cut off function such that $\zeta=1$ on $(-r, r)$ and $\left|\zeta^{\prime}\right| \leq 2 r^{-1} .$ Noting that by $r \leq 2 R$, it holds $U_{1} \subset U_{2}$ and
$\bm{v}_{2}=\bm{v}_{1} \zeta \in W^{1,2}\left(U_{2}\right) .$

Set $\bm{v}_{3}(x)=\bm{v}_{2}(Rx),\ x\in U_{3}$. There is an Sobolev extension $\tilde{\bm{v}}_{3}(x) \in W^{1,2}(\R^3)  $ with $\tilde{\bm{v}}_{3}(x)=\bm{v}_{3}(x),\ x\in U_{3}$ and
\begin{equation}\label{2.6-1}
\|\tilde{\bm{v}}_{3}\|_{L^{2}(\R^3)}\leq C \|\bm{v}_{3}\|_{L^{2}(U_{3})},\quad \|\nabla\tilde{\bm{v}}_{3}\|_{L^{2}(\R^3)}\leq C \|\bm{v}_{3}\|_{W^{1,2}(U_{3})}.
\end{equation}
By means of Sobolev's inequality and a simple scaling argument along with \eqref{2.5-1} and \eqref{2.6-1}, we get
\begin{align*}
\|\bm{u}\|_{L^{q_{v}}_{v}L^{q_{h}}_{h}\left(U\right)} & \leq\|\bm{v}_{2}\|_{L^{q_{v}}_{v}L^{q_{h}}_{h}\left(U_{2}\right)} \\
&\leq R^{\frac{2}{q_{h}}+\frac{1}{q_{v}}} \|\bm{v}_{3}\|_{L^{q_{v}}_{v}L^{q_{h}}_{h}\left(U_{3}\right)}\\
&\leq C R^{\frac{2}{q_{h}}+\frac{1}{q_{v}}} \|\tilde{\bm{v}}_{3}\|_{L^{q_{v}}_{v}L^{q_{h}}_{h}\left(\R^3\right)}\\
&\leq C R^{\frac{2}{q_{h}}+\frac{1}{q_{v}}} \|\tilde{\bm{v}}_{3}\|_{L^{2}\left(\R^3\right)}^{\frac{2}{q_{h}}+\frac{1}{q_{v}}-\frac{1}{2}}\|\nabla\tilde{\bm{v}}_{3}\|_{L^{2}\left(\R^3\right)}^{\frac{3}{2}-\frac{2}{q_{h}}-\frac{1}{q_{v}}}\\
&\leq C R^{\frac{2}{q_{h}}+\frac{1}{q_{v}}}\left(\|\bm{\bm{v}}_{3}\|_{L^2(U_{3})}+\|\bm{v}_{3}\|_{L^2(U_{3})}^{\frac{2}{q_{h}}+\frac{1}{q_{v}}-\frac{1}{2}}\|\nabla \bm{v}_{3}\|_{L^2(U_{3})}^{\frac{3}{2}-\frac{2}{q_{h}}-\frac{1}{q_{v}}}\right)\\
&\leq C \|\bm{v}_{2}\|_{L^2(U_{2})}+C\|\bm{v}_{2}\|_{L^2(U_{2})}^{\frac{2}{q_{h}}+\frac{1}{q_{v}}-\frac{1}{2}}\|\nabla \bm{v}_{2}\|_{L^2(U_{2})}^{\frac{3}{2}-\frac{2}{q_{h}}-\frac{1}{q_{v}}}\\
& \leq C r^{-\frac{3}{2}+\frac{2}{q_{h}}+\frac{1}{q_{v}}}\|\bm{u}\|_{L^{2}(U)}+C\|\bm{u} \|_{L^{2}(U)}^{\frac{2}{q_{h}}+\frac{1}{q_{v}}-\frac{1}{2}}\|\nabla \bm{u}\|_{L^{2}(U)}^{\frac{3}{2}-\frac{2}{q_{h}}-\frac{1}{q_{v}}},
\end{align*}
which gives rise to \eqref{2.4}.   By using of H\"{o}lder's inequality and \eqref{2.4} with $r=r_{n}$, we achieve
\begin{equation*}
\|\bm{u}\|_{L^{p}\left(-r_{n}^{2}, 0 ;L^{q_{v}}_{v}L^{q_{h}}_{h}\left(U_{n}(R)\right)\right)}^{2}  \leq C\|\bm{u}\|_{L^{\infty}\left(-r_{n}^{2}, 0 ; L^{2}(U_{n})\right)}^{2}+C\|\nabla \bm{u}\|_{L^{2}\left(-r_{n}^{2}, 0 ; L^{2}(U_{n})\right)}^{2}.
\end{equation*}
This completes the proof of this lemma.
\end{proof}

\section{Proof of the Main Results}

In this section, we apply a similar arguement in  Cafferalli-Kohn-Nirenberg \cite{Caffarelli1982}, Chae-Wolf \cite{Chae2019}, or Wang-Wu-Zhang \cite{Wang2020} to prove Theorem \ref{thm2}.

We assume that solution $\bm u$ satisfies
\begin{equation}
\partial_{3} \bm{u} \in L^{p_{0},1}\left(-1,0;L^{q_{0}}(B(2))\right),~\frac{2}{p_{0}}+\frac{3}{q_{0}}=2,~\frac{3}{2} < q_{0}<+ \infty.
\end{equation}
For fixed $p_{0},q_{0}$, we can pick $\frac{4q_{0}}{4q_{0}-5}< p_{0}^*<p_{0}$. Set
\begin{align}\label{3.2}
\B_{i}=r_{i}^{2-\frac{2}{p_{0}^*}-\frac{3}{q_{0}}}\|\partial_{3} \bm{u}\|_{L^{p_{0}^*}\left(-r_{i}^2,0;L^{q_{0}}(B(2))\right)}.
\end{align}
By Lemma \ref{A5}, we have
\begin{align}\label{B}
\sum_{i=0}^{+\infty} \B_{i}\leq C \left\|\partial_{3}\bm{u}\right\|_{L^{p_{0}, 1}\left(-1,0;L^{q_{0}}(B(2))\right) }.
\end{align}
Let $\eta\left(x_{3}, t\right) \in C_{c}^{\infty}((-1,1) \times(-1,0])$ denotes a cut-off function,  $0 \leq \eta \leq 1$, and $\eta=1$ on $\left(-\frac{1}{2}, \frac{1}{2}\right) \times\left(-\frac{1}{4}, 0\right)$.

In addition, let $\frac{1}{2} \leq \rho<R\leq1$ be arbitrarily chosen, but $|R-\rho|\leq\frac{1}{2} .$ Let $\psi=\psi\left(x^{\prime}\right) \in C^{\infty}\left(\R^{2}\right)$ with $0 \leq \psi \leq 1$ in $B^{\prime}(R)$ satisfying
\begin{equation}
\psi(x^{\prime})=\psi(|x^{\prime}|)=\left\{\begin{array}{l}
1 \text { in } B^{\prime}(\rho) \\
0 \text { in } \R^{2} \backslash B^{\prime}\left(\frac{R+\rho}{2}\right)
\end{array}\right.
\end{equation}
and
\begin{equation*}
|D \psi| \leq \frac{C}{R-\rho}, \quad\left|D^{2} \psi\right| \leq \frac{C}{(R-\rho)^{2}}.
\end{equation*}

For $j=0,1,\cdots,n$, denote $\chi_{j} =\chi_{B^{\prime}(R)}(x^{\prime})\cdot\eta(2^{j}\cdot x_{3},2^{2j}\cdot t)$, where $\chi_{B^{\prime}(R)}$ is the indicator function of the set $B^{\prime}(R)$. Let
\begin{equation}\label{phi}
\begin{array}{lll}
\phi_{j}=\left\{\begin{array}{ll}
\chi_{j}-\chi_{j+1}, & \text { if } \quad j=0, \ldots, n-1, \\
\chi_{n}, & \text { if } \quad j=n,
\end{array}\right.
\end{array}
\end{equation}
and the mean value function
\begin{equation} \label{overline}
\overline{(u_{3})}_{k}(x^{\prime})=\frac{1}{2r_{k}}
\int_{-r_{k}}^{r_{k}}u_{3}(x^{\prime},\omega)~\text{d}\omega.
\end{equation}

Taking the test function $\varphi=\Phi_n\eta\psi$ in \eqref{local energy}, it yields that
\begin{align}\label{key}
& \int_{U_{0}(R)}|\bm{u}(x, t)|^{2} \Phi_{n} \eta \psi\dx  +2\int_{-1}^{t} \int_{U_{0}(R)}|\nabla \bm{u}|^{2} \Phi_{n} \eta \psi\dx  \text{d}s   \notag\\
\leq&~  \int_{-1}^{t} \int_{U_{0}(R)}|\bm{u}|^{2}\left(\partial_{s}+\Delta\right)\left(\Phi_{n} \eta \psi\right)\dx  \text{d}s  + \int_{-1}^{t} \int_{U_{0}(R)}|\bm{u}|^{2} \bm{u} \cdot \nabla\left(\Phi_{n} \eta \psi\right)\dx  \text{d}s   \notag\\
&+2\int_{-1}^{t} \int_{U_{0}(R)} \pi \bm{u} \cdot \nabla\left(\Phi_{n} \eta \psi\right)\dx  \text{d}s  .
\end{align}
Next, we shall handle the right side of \eqref{key} term by term.
\subsection{Estimates for nonlinear terms}\label{section3.1}
\begin{lem}\label{lem1}
{\sl Let $(\bm{u}, \pi)$ be a suitable weak solution of \eqref{NS} in $\R^{3} \times(-1,0) .$ If $(\bm{u},\pi)$ satisfies the assumption of Theorem \ref{thm2}. Then we have
\begin{equation}\label{2.5}
\int_{-1}^{t} \int_{U_{0}(R)}|\bm{u}|^{2}\left(\partial_{s}+\Delta\right)\left(\Phi_{n} \eta \psi\right)\dx   \normalfont{\text{d}s}   \leq C \frac{\E}{(R-\rho)^{2}}.
\end{equation}}
\end{lem}
\begin{proof}
  The proof of this lemma is similar to Lemma 3.1 in \cite{Wang2020}, referring to the properties of $\Phi_n,\eta$ and $\psi$. We omit it here.
\end{proof}
\begin{lem}\label{lem2}
{\sl Under the assumption of Lemma \ref{lem1}, we have
\begin{align}\label{2.6}
&\int_{-1}^{t} \int_{U_{0}(R)}|\bm{u}|^{2} \bm{u} \cdot \nabla\left(\Phi_{n} \eta \psi\right)\dx  \normalfont{\text{d}s}   \notag\\
\leq& C\sum_{i=0}^{n} \B_{i}  \left(r_{i}^{-1} E_{i}(R)\right)+C(R-\rho)^{-1} \ \E^{\frac{1}{2}} \sum_{i=0}^{n} r_{i}^{\frac{1}{2}}\left(r_{i}^{-1} E_{i}(R)\right),
\end{align}
where $\B_{i}$ is defined in \eqref{3.2}}.
\end{lem}
\begin{proof}
We first note  that
\begin{align}\label{I}
&\int_{-1}^{t} \int_{U_{0}(R)}|\bm{u}|^{2} \bm{u} \cdot \nabla\left(\Phi_{n} \eta \psi\right)\dx  \text{d}s   \notag\\
=&\int_{-1}^{t} \int_{U_{0}(R)}|\bm{u}|^2 u_{3}\cdot\partial_{3} \left(\Phi_{n} \eta\right) \psi\dx  \text{d}s   +\sum_{\mu=1,2}\int_{-1}^{t} \int_{U_{0}(R)}|\bm{u}|^{2} u_{\mu} \cdot\Phi_{n} \eta\partial_{\mu}\psi\dx  \text{d}s  \notag\\
\eqdefa &I_{1}+I_{2}.
\end{align}

By integration by parts, the estimates \eqref{Phi_{n}} for $\Phi_{n}$, H\"{o}lder inequality, Lemma \ref{energy} and Lemma \ref{A3}, we have
\begin{align}\label{I_{1}}
I_{1}=& \sum_{k=0}^{n}  \int_{Q_{0}(R)}\bm{u}\cdot\bm{u}\cdot u_{3}\cdot\partial_{3} \left(\Phi_{n} \phi_{k}\right) \psi\dx  \text{d}s  \notag\\
=&\sum_{k=0}^{n}  \int_{Q_{0}(R)}\bm{u}\cdot\bm{u}\  \left(u_{3}-\overline{(u_{3})}_{k}\right)\cdot\partial_{3} \left(\Phi_{n} \phi_{k}\right) \psi\dx  \text{d}s  \notag\\
&-2\sum_{k=0}^{n}  \int_{Q_{0}(R)}\partial_{3}\bm{u}\cdot\bm{u}\cdot \overline{(u_{3})}_{k}\cdot \left(\Phi_{n} \phi_{k}\right) \psi\dx  \text{d}s  \notag\\
\lesssim&\sum_{k=0}^{n}  r_{k}^{-2} \int_{-r_{k}^2}^{0}\|\bm{u}\|_{L^{\frac{2q_{0}}{q_{0}-1}}(U_{k}(R))}\|\bm{u}\|_{L_{v}^{2}L_{h}^{\frac{2q_{0}}{q_{0}-1}}(U_{k}(R))}\|u_{3}-\overline{(u_{3})}_{k}\|_{L_{v}^{2q_{0}}L_{h}^{q_{0}}(U_{k}(R))} \text{d}s  \notag\\
&+\sum_{k=0}^{n}  r_{k}^{-1} \int_{-r_{k}^2}^{0}\|\partial_{3}\bm{u}\|_{L^{q_{0}}\left(U_{k}(R)\right)} \|\bm{u}\|_{L^{\frac{2q_{0}}{q_{0}-1}}(U_{k}(R))} \|\overline{(u_{3})}_{k}\|_{L^{\frac{2q_{0}}{q_{0}-1}}(U_{k}(R))}   \text{d}s   \notag\\
\lesssim& \sum_{k=0}^{n} r_{k}^{-1-\frac{1}{2q_{0}}}\int_{-r_{k}^2}^{0}\|\bm{u}\|_{L^{\frac{2q_{0}}{q_{0}-1}}(U_{k}(R))}\|\bm{u}\|_{L_{v}^{2}L_{h}^{\frac{2q_{0}}{q_{0}-1}}(U_{k}(R))} \|\partial_{3}\bm{u}\|_{L^{q_{0}}(U_{k}(R))} \text{d}s  \notag\\
\lesssim& \sum_{k=0}^{n} r_{k}^{1-\frac{2}{p_{0}^*}-\frac{3}{q_{0}}}\|\bm{u}\|_{L^{\frac{4q_{0}}{3}}\left(-r_{k}^2,0;L^{\frac{2q_{0}}{q_{0}-1}}(U_{k}(R))\right)}\|\bm{u}\|_{L^{2q_{0}}\left(-r_{k}^2,0;L_{v}^{2}L_{h}^{\frac{2q_{0}}{q_{0}-1}}(U_{k}(R))\right)}\notag\\
&\quad\times \|\partial_{3}\bm{u}\|_{L^{p_{0}^*}\left(-r_{k}^2,0;L^{q_{0}}(U_{k}(R))\right)}\notag\\
\lesssim&\sum_{i=0}^{n} \B_{i} \left(r_{i}^{-1} E_{i}(R)\right).
\end{align}
For the second term, by \eqref{Phi_{n}} and Lemma \ref{energy} again, we have
\begin{align}\label{I_{2}}
I_{2} \lesssim& \frac{1}{(R-\rho)} \sum_{i=0}^{n} r_{i}^{-1} \int_{Q_{i}(R)}|\bm{u}|^{3} \dx \text{d}s \notag\\
\lesssim& \frac{1}{(R-\rho)} \sum_{i=0}^{n} r_{i}^{-1} r_{i}^{1 / 2}\|\bm{u}\|_{L^{4}\left(-r_{i}^{2}, 0 ; L^{3}\left(U_{i}(R)\right)\right)}^{3} \notag\\
\lesssim&\frac{\E^{\frac{1}{2}}}{(R-\rho)}  \sum_{i=0}^{n} r_{i}^{\frac{1}{2}}\left(r_{i}^{-1} E_{i}(R)\right).
\end{align}
Combining \eqref{I_{1}} with \eqref{I_{2}},  we obtain \eqref{2.6}.
\end{proof}
\subsection{Estimates for the pressure}\label{section3.2}

This part is devoted to the estimates regarding the
third term on the right side of \eqref{key}. Compared with \cite{Chae2019, Wang2020}, the decomposition of the pressure is in a slightly different way.

Given $3\times3$ matrix valued function $f=\left(f_{\mu\nu}\right)$, we set
\begin{align*}
\widehat{\J(f)}(\xi)= \sum_{\mu,\nu=1,2,3} \frac{\xi_{\mu}\xi_{\nu}}{|\xi|^2}\mathcal{F} \left(f_{\mu\nu}\cdot \chi_{U_{0}(R)}\right),
\end{align*}
where the Fourier transform is defined by
\begin{equation*}
\mathcal{F} f(\xi)=\hat{f}(\xi)=\int_{\R^3} e^{-i(x \cdot \xi)} f(x)\dx  .
\end{equation*}
Therefore, $\J: L_{v}^{q_{v}}L_{h}^{q_{h}}\left(U_{0}(R)\right) \rightarrow L_{v}^{q_{v}}L_{h}^{q_{h}}\left(\R^{3}\right)$,  $1<q_{h},q_{v}<+\infty$ defines a bounded linear operator \cite{Lizorkin1970}, with
\begin{align}\label{ani}
\|\J(f)\|_{L_{v}^{q_{v}}L_{h}^{q_{h}}\left(\R^{3}\right)} \leq C\|f\|_{L_{v}^{q_{v}}L_{h}^{q_{h}}\left(U_{0}(R)\right)}.
\end{align}

Denote
\begin{align}\label{fg}
f=\bm{u}\otimes\bm{u}\cdot \chi_{0},
\end{align}
and
\begin{align}\label{pitau}
\pi_{0}=\J(f),\ \pi_{h}=\pi-\pi_{0}.
\end{align}
It follows that
\begin{align*}
-\Delta \pi_{0}=\nabla \cdot \nabla \cdot f, \quad \text { in } \quad \R^{3} \times(-1,0)
\end{align*}
in the sense of distributions and  $\pi_{h}$  is harmonic in $Q_{1}(R)$. Then we have
\begin{align}\label{pressure}
&\int_{-1}^{t} \int_{U_{0}(R)} \pi \bm{u} \cdot \nabla\left(\Phi_{n} \eta \psi\right)\dx  \text{d}s   \notag\\
=&\int_{-1}^{t} \int_{U_{0}(R)} \pi_{0} \bm{u} \cdot \nabla\left(\Phi_{n} \eta \psi\right)\dx  \text{d}s  +\int_{-1}^{t} \int_{U_{0}(R)} \pi_{h} \bm{u} \cdot \nabla\left(\Phi_{n} \eta \psi\right)\dx  \text{d}s   \notag\\
=&\int_{-1}^{t} \int_{U_{0}(R)} \pi_{0} u_{3} \cdot \partial_{3}\left(\Phi_{n} \eta \psi\right)\dx  \text{d}s  +\sum_{\mu=1,2}\int_{-1}^{t} \int_{U_{0}(R)} \pi_{0} u_{\mu} \cdot \Phi_{n} \eta \partial_{\mu}\psi\dx  \text{d}s  \notag\\
&-\int_{-1}^{t} \int_{U_{0}(R)} \nabla\pi_{h}\cdot \bm{u}\cdot  \left(\Phi_{n} \eta \psi\right)\dx  \text{d}s  \notag\\
\eqdefa&J+K+H.
\end{align}
As to $J$, setting $\tau_{0}=\J\left(\partial_{3}\left(\bm{u}\otimes\bm{u}\right)\cdot\chi_{0}\right)$ and $\tau_{h}=\J\left(\bm{u}\otimes\bm{u}\cdot\partial_{3}\chi_{0}\right)$, we have $\partial_{3}\pi_{0}=\tau_{0}+\tau_{h}$ and
\begin{align}\label{J}
J=&\sum_{k=0}^{n}\int_{-1}^{t} \int_{U_{0}(R)} \pi_{0}\cdot u_{3} \cdot \partial_{3}\left(\Phi_{n} \phi_{k} \psi\right)\dx  \text{d}s      \notag\\
=&-\sum_{k=4}^{n}\int_{-1}^{t} \int_{U_{0}(R)} \left(\pi_{0}\cdot \partial_{3}u_{3}+\tau_{0}\cdot u_{3}\right) \cdot \left(\Phi_{n} \phi_{k} \psi\right)\dx  \text{d}s  \notag\\
&-\sum_{k=4}^{n}\int_{-1}^{t} \int_{U_{0}(R)} \tau_{h}\cdot u_{3} \cdot \left(\Phi_{n} \phi_{k} \psi\right)\dx  \text{d}s  \notag\\
&+\sum_{k=0}^{3}\int_{-1}^{t} \int_{U_{0}(R)} \pi_{0}\cdot u_{3} \cdot \partial_{3}\left(\Phi_{n} \phi_{k} \psi\right)\dx  \text{d}s  \notag\\
\eqdefa &J_{1}+J_{2}+J_{3}.
\end{align}
We will present the estimates of $J_{i},i=1,2,3$, $K$ and $H$ in the following several lemmas.
\begin{lem}\label{lem3} {\sl Under the assumption of Lemma \ref{lem1}, we have
	\begin{align}\label{J_{1}}
	J_{1}\leq C\sum_{i=0}^{n} \B_{i} \left(r_{i}^{-1} E_{i}(R)\right).
	\end{align}}
\end{lem}
\begin{proof}
Let $\pi_{0,j}=\J\left(\bm{u}\otimes\bm{u}\cdot\phi_{j}\right)$ and $	\tau_{0,j}=\J\left(\partial_{3}\left(\bm{u}\otimes\bm{u}\right)\cdot\phi_{j}\right)$. By integration by parts, we have
\begin{align}\label{J_{1}'}
J_{1}=&-\sum_{k=4}^{n}\sum_{j=0}^{n}\int_{-1}^{t} \int_{U_{0}(R)} \left(\pi_{0,j}\cdot \partial_{3}u_{3}+\tau_{0,j}\cdot u_{3}\right) \cdot \left(\Phi_{n} \phi_{k} \psi\right)\dx  \text{d}s  \notag\\
=&-\sum_{k=4}^{n}\sum_{j=k-3}^{n}\int_{-1}^{t} \int_{U_{0}(R)} \left(\pi_{0,j}\cdot \partial_{3}u_{3}+\tau_{0,j}\cdot u_{3}\right) \cdot \left(\Phi_{n} \phi_{k} \psi\right)\dx  \text{d}s  \notag\\
&-\sum_{k=4}^{n}\sum_{j=0}^{k-4}\int_{-1}^{t} \int_{U_{0}(R)} \left(\pi_{0,j}\cdot \partial_{3}u_{3}+\tau_{0,j}\cdot u_{3}\right) \cdot \left(\Phi_{n} \phi_{k} \psi\right)\dx  \text{d}s  \notag\\
=&-\sum_{k=4}^{n}\int_{-1}^{t} \int_{U_{0}(R)} \J\left(\bm{u}\otimes\bm{u}\cdot\chi_{k-3}\right)\cdot \partial_{3}u_{3} \cdot \left(\Phi_{n} \phi_{k} \psi\right)\dx  \text{d}s  \notag\\
&-\sum_{k=4}^{n}\int_{-1}^{t} \int_{U_{0}(R)} \J\left(\partial_{3}\left(\bm{u}\otimes\bm{u}\right)\cdot\chi_{k-3}\right)\cdot u_{3} \cdot \left(\Phi_{n} \phi_{k} \psi\right)\dx  \text{d}s  \notag\\
&-\sum_{j=0}^{n-4}\sum_{k=j+4}^{n}\int_{-1}^{t} \int_{U_{0}(R)} \left(\pi_{0,j}\cdot \partial_{3}u_{3}+\tau_{0,j}\cdot u_{3}\right) \cdot \left(\Phi_{n} \phi_{k} \psi\right)\dx  \text{d}s  \notag\\
=&\sum_{k=4}^{n}\int_{-1}^{t} \int_{U_{0}(R)} \J\left(\bm{u}\otimes\bm{u}\cdot\chi_{k-3}\right)\cdot \left(u_{3}-\overline{(u_{3})}_{k}\right) \cdot \partial_{3}\left(\Phi_{n} \phi_{k} \psi\right)\dx  \text{d}s  \notag\\
&+\sum_{k=4}^{n}\int_{-1}^{t} \int_{U_{0}(R)} \J\left(\bm{u}\otimes\bm{u}\cdot\partial_{3}\chi_{k-3}\right)\cdot \left(u_{3}-\overline{(u_{3})}_{k}\right) \cdot \left(\Phi_{n} \phi_{k} \psi\right)\dx  \text{d}s  \notag\\
&-\sum_{k=4}^{n}\int_{-1}^{t} \int_{U_{0}(R)} \J\left(\partial_{3}\left(\bm{u}\otimes\bm{u}\right)
\cdot\chi_{k-3}\right)\cdot\overline{(u_{3})}_{k} \cdot \left(\Phi_{n} \phi_{k} \psi\right)\dx  \text{d}s  \notag\\
&-\sum_{j=0}^{n-4}\sum_{k=j+4}^{n}\int_{-1}^{t} \int_{U_{0}(R)} \pi_{0,j}\cdot \partial_{3}u_{3} \cdot \left(\Phi_{n} \phi_{k} \psi\right)\dx  \text{d}s  \notag\\
&-\sum_{j=0}^{n-4}\sum_{k=j+4}^{n}\int_{-1}^{t} \int_{U_{0}(R)} \tau_{0,j}\cdot  \left(u_{3}-\overline{(u_{3})}_{j}\right) \cdot \left(\Phi_{n} \phi_{k} \psi\right)\dx  \text{d}s  \notag\\
&-\sum_{j=0}^{n-4}\sum_{k=j+4}^{n}\int_{-1}^{t} \int_{U_{0}(R)} \tau_{0,j}\cdot \overline{(u_{3})}_{j} \cdot \left(\Phi_{n} \phi_{k} \psi\right)\dx  \text{d}s  \notag\\
\eqdefa&J_{11}+J_{12}+J_{13}+J_{14}+J_{15}+J_{16}.
\end{align}
By \eqref{Phi_{n}}, \eqref{ani}, Lemma \ref{energy} and Lemma \ref{A3}, we have
\begin{align}\label{J_{11}}
J_{11}\lesssim&\sum_{k=4}^{n}r_{k}^{-2}\int_{-r_{k}^2}^{0}\|\J\left(\bm{u}\otimes\bm{u}\cdot\chi_{k-3}\right)\|_{L_{v}^{\frac{2q_{0}}{2q_{0}-1}}L_{h}^{\frac{q_{0}}{q_{0}-1}}(U_{k}(R))}\|u_{3}-\overline{(u_{3})}_{k}\|_{L_{v}^{2q_{0}}L_{h}^{q_{0}}(U_{k}(R))} \text{d}s  \notag\\
\lesssim& \sum_{k=0}^{n} r_{k}^{-1-\frac{1}{2q_{0}}}\int_{-r_{k}^2}^{0}\|\bm{u}\|_{L^{\frac{2q_{0}}{q_{0}-1}}(U_{k}(R))}\|\bm{u}\|_{L_{v}^{2}L_{h}^{\frac{2q_{0}}{q_{0}-1}}(U_{k}(R))} \|\partial_{3}\bm{u}\|_{L^{q_{0}}(U_{k}(R))} \text{d}s  \notag\\
\lesssim&\sum_{i=0}^{n} \B_{i} \left(r_{i}^{-1} E_{i}(R)\right),
\end{align}
which is analogous to \eqref{I_{1}}. Similarly, for $J_{12},J_{13}$, we have
\begin{align}\label{J_{12}}
J_{12}\lesssim&\sum_{k=4}^{n}r_{k}^{-1}\int_{-r_{k}^2}^{0}\|\J
\left(\bm{u}\otimes\bm{u}\cdot\partial_{3}\chi_{k-3}\right)
\|_{L_{v}^{\frac{2q_{0}}{2q_{0}-1}}L_{h}^{\frac{q_{0}}{q_{0}-1}}(U_{k}(R))}
\|u_{3}-\overline{(u_{3})}_{k}\|_{L_{v}^{2q_{0}}L_{h}^{q_{0}}(U_{k}(R))} \text{d}s  \notag\\
\lesssim& \sum_{k=0}^{n} r_{k}^{-1-\frac{1}{2q_{0}}}\int_{-r_{k}^2}^{0}\|\bm{u}\|_{L^{\frac{2q_{0}}
{q_{0}-1}}(U_{k}(R))}\|\bm{u}\|_{L_{v}^{2}L_{h}^{\frac{2q_{0}}{q_{0}-1}}(U_{k}(R))} \|\partial_{3}\bm{u}\|_{L^{q_{0}}(U_{k}(R))} \text{d}s  \notag\\
\lesssim&\sum_{i=0}^{n} \B_{i} \left(r_{i}^{-1} E_{i}(R)\right),
\end{align}
and
\begin{align}\label{J_{13}}
J_{13}\lesssim&\sum_{k=4}^{n}r_{k}^{-1}\int_{-r_{k}^2}^{0}
\|\J\left(\partial_{3}\left(\bm{u}\otimes\bm{u}\right)\cdot\chi_{k-3}\right)
\|_{L^{\frac{2q_{0}}{q_{0}+1}}(U_{k}(R))}\|\overline{(u_{3})}_{k}
\|_{L^{\frac{2q_{0}}{q_{0}-1}}(U_{k}(R))} \text{d}s  \notag\\
\lesssim& \sum_{k=0}^{n} r_{k}^{-1-\frac{1}{2q_{0}}}\int_{-r_{k}^2}^{0}
\|\bm{u}\|_{L^{\frac{2q_{0}}{q_{0}-1}}(U_{k}(R))}
\|\bm{u}\|_{L_{v}^{2}L_{h}^{\frac{2q_{0}}{q_{0}-1}}(U_{k}(R))} \|\partial_{3}\bm{u}\|_{L^{q_{0}}(U_{k}(R))} \text{d}s  \notag\\
\lesssim&\sum_{i=0}^{n} \B_{i} \left(r_{i}^{-1} E_{i}(R)\right).
\end{align}
By \eqref{Phi_{n}}, \eqref{pi_{0,j}} and Lemma \ref{energy}, we have
\begin{align}\label{J_{14}}
J_{14}=&\sum_{j=0}^{n-4}\sum_{k=j+4}^{n} r_{k}^{-1}\int_{-r_{k}^2}^{0}\|\pi_{0,j}\|_{L^{\frac{q_{0}}{q_{0}-1}}(U_{k}(R))}\|\partial_{3}u_{3}\|_{L^{q_{0}}(U_{k}(R))} \text{d}s  \notag\\
\lesssim&\sum_{j=0}^{n-4}\sum_{k=j+4}^{n}r_{j}^{-1-\frac{1}{2q_{0}}} r_{k}^{-\frac{1}{q_{0}}}\int_{-r_{k}^2}^{0}\|\bm{u}\|_{L^{\frac{4q_{0}}{2q_{0}-1}}(U_{j}(R))}^2\|\partial_{3}\bm{u}\|_{L^{q_{0}}(U_{k}(R))} \text{d}s  \notag\\
\lesssim&\sum_{j=0}^{n-4}\sum_{k=j+4}^{n}r_{j}^{-1-\frac{1}{2q_{0}}} r_{k}^{2-\frac{5}{2q_{0}}-\frac{2}{p_{0}^*}}\|\bm{u}\|_{L^{\frac{8q_{0}}{3}}\left(-r_{k}^2,0;L^{\frac{4q_{0}}{2q_{0}-1}}(U_{j}(R))\right)}^2\nonumber\\
&\quad\times\|\partial_{3}\bm{u}\|_{L^{p_{0}^*}\left(-r_{k}^2,0;L^{q_{0}}(U_{k}(R))\right)}\notag\\
\lesssim&\sum_{i=0}^{n} \B_{i} \left(r_{i}^{-1} E_{i}(R)\right).
\end{align}
By \eqref{Phi_{n}}, \eqref{tau_{0,j}}, Lemma \ref{energy} and Lemma \ref{A3}, we have
\begin{align}\label{J_{15}}
J_{15}\lesssim&\sum_{j=0}^{n-4}\sum_{k=j+4}^{n} r_{k}^{-1}\int_{-r_{k}^2}^{0}\|\tau_{0,j}\|_{L^{\frac{q_{0}}{q_{0}-1}}(U_{k}(R))}\|u_{3}-\overline{(u_{3})}_{j}\|_{L^{q_{0}}(U_{k}(R))} \text{d}s  \notag\\
\lesssim&\sum_{j=0}^{n-4}\sum_{k=j+4}^{n}r_{j}^{-1-\frac{3}{2q_{0}}} \int_{-r_{k}^2}^{0}\|\bm{u}\|_{L^{\frac{4q_{0}}{2q_{0}-1}}(U_{j}(R))}^2\|\partial_{3}\bm{u}\|_{L^{q_{0}}(U_{j}(R))} \text{d}s  \notag\\
\lesssim&\sum_{j=0}^{n-4}\sum_{k=j+4}^{n}r_{j}^{-1-\frac{3}{2q_{0}}} r_{k}^{2-\frac{3}{2q_{0}}-\frac{2}{p_{0}^*}}\|\bm{u}\|_{L^{\frac{8q_{0}}{3}}\left(-r_{k}^2,0;L^{\frac{4q_{0}}{2q_{0}-1}}(U_{j}(R))\right)}^2\nonumber\\
&\quad\quad\qquad\times\|\partial_{3}\bm{u}\|_{L^{p_{0}^*}\left(-r_{k}^2,0;L^{q_{0}}(U_{j}(R))\right)}
\notag\\
\lesssim&\sum_{i=0}^{n} \B_{i} \left(r_{i}^{-1} E_{i}(R)\right).
\end{align}
By \eqref{Phi_{n}}, \eqref{tau'_{0,j}}, Lemma \ref{energy} and Lemma \ref{A3}, we have
\begin{align}\label{J_{16}}
J_{16}\lesssim&\sum_{j=0}^{n-4}\sum_{k=j+4}^{n} r_{k}^{-1}\int_{-r_{k}^2}^{0}\|\tau_{0,j}\|_{L^{2}(U_{k}(R))}\|\overline{(u_{3})}_{j}\|_{L^{2}(U_{k}(R))} \text{d}s  \notag\\
\lesssim&\sum_{j=0}^{n-4}\sum_{k=j+4}^{n}r_{j}^{-1-\frac{3}{2q_{0}}} \int_{-r_{k}^2}^{0}\|\partial_{3}\bm{u}\otimes\bm{u}\|_{L^{\frac{2q_{0}}{q_{0}+1}}(U_{j}(R))}\|\bm{u}\|_{L^{2}(U_{j}(R))} \text{d}s  \notag\\
\lesssim&\sum_{j=0}^{n-4}\sum_{k=j+4}^{n}r_{j}^{-1-\frac{3}{2q_{0}}} \int_{-r_{k}^2}^{0}\|\bm{u}\|_{L^{\frac{2q_{0}}{q_{0}-1}}(U_{j}(R))}
\|\partial_{3}\bm{u}\|_{L^{q_{0}}(U_{j}(R))}\|\bm{u}\|_{L^{2}(U_{j}(R))} \text{d}s  \notag\\
\lesssim&\sum_{j=0}^{n-4}\sum_{k=j+4}^{n}r_{j}^{-1-\frac{3}{2q_{0}}} r_{k}^{2-\frac{3}{2q_{0}}-\frac{2}{p_{0}^*}}\|\bm{u}\|_{L^{\frac{4q_{0}}{3}}\left(-r_{k}^2,0;L^{\frac{2q_{0}}{q_{0}-1}}(U_{j}(R))\right)}\|\bm{u}\|_{L^{\infty}\left(-r_{k}^2,0;L^{2}(U_{j}(R))\right)}\notag\\
&\qquad\quad\quad\times\|\partial_{3}\bm{u}\|_{L^{p_{0}^*}\left(-r_{k}^2,0;L^{q_{0}}(U_{j}(R))\right)}\notag\\
\lesssim&\sum_{i=0}^{n} \B_{i} \left(r_{i}^{-1} E_{i}(R)\right).
\end{align}
Summing up all the estimates of \eqref{J_{1}'}, \eqref{J_{11}}, \eqref{J_{12}}, \eqref{J_{13}}, \eqref{J_{14}}, \eqref{J_{15}} and \eqref{J_{16}}, we have \eqref{J_{1}}.
\end{proof}
\begin{lem}\label{lem4} {\sl Under the assumption of Lemma \ref{lem1}, we have
\begin{equation}\label{H}
J_{2}+J_{3}+H\leq \frac{C}{(R-\rho)^{\frac{5}{2}}} \E^{\frac{3}{2}}.
\end{equation}}
\end{lem}
\begin{proof}
	The proof of this lemma is similar with Lemma 3.3 in \cite{Wang2020}.
	By using of integration by parts  and  $\nabla \cdot \bm{u}=0$, we have
	\begin{align}\label{3.25}
	H=&-\sum_{k=0}^{3}\int_{Q_{0}(R)} \nabla\pi_{h}\cdot \bm{u}\cdot  \left(\Phi_{n} \phi_{k} \psi\right)\dx  \text{d}s  -\sum_{k=4}^{n}\int_{Q_{0}(R)} \nabla\pi_{h}\cdot \bm{u}\cdot  \left(\Phi_{n} \phi_{k} \psi\right)\dx  \text{d}s  \notag\\
	=&\sum_{k=0}^{3}\int_{Q_{0}(R)} \pi_{h}\cdot \bm{u}\cdot  \nabla\left(\Phi_{n} \phi_{k} \psi\right)\dx  \text{d}s  -\sum_{k=4}^{n}\int_{Q_{0}(R)} \nabla\pi_{h}\cdot \bm{u}\cdot  \left(\Phi_{n} \phi_{k} \psi\right)\dx  \text{d}s  \notag\\
	\eqdefa&H_{1}+H_{2}.
	\end{align}
	Applying \eqref{Phi_{n}} and H\"{o}lder's inequality, we have
	\begin{align}\label{3.26}
	J_{3}+H_{1}\lesssim&\int_{Q_{0}(R)}|\pi_{0}||\bm{u}|\dx  \text{d}s  +\frac{1}{R-\rho}\int_{Q_{0}(R)}|\pi_{h}||\bm{u}|\dx  \text{d}s  \notag\\
	\lesssim& \frac{1}{R-\rho}\|\bm{u}\|_{L^{3}\left(\R^3\times(-1,0)\right)}^3\notag\\
	\lesssim&\frac{1}{R-\rho} \E^{\frac{3}{2}}.
	\end{align}
	Moreover, we may choose a finite family of points $\left\{x_{\nu}^{\prime}\right\}$ in $ B^{\prime}\left(\frac{R+\rho}{2}\right)$ such that $\left\{B^{\prime}(x_{\nu}^{\prime};\frac{R-\rho}{4})\right\}$  cover the ball $ B^{\prime}\left(\frac{R+\rho}{2}\right)$ and
	\begin{align}
	\sum_{\nu}^{}\chi_{B^{\prime}(x_{\nu}^{\prime};\frac{R-\rho}{2})} \leq C.
	\end{align}
	\begin{align}\label{5.6}
	J_{2}+H_{2}\lesssim&  \sum_{k=4}^{n} \sum_{\nu} r_{k}^{-1} \int_{Q_{k}(R) \cap\left\{\left|x^{\prime}-x_{\nu}^{\prime}\right|<\frac{R-\rho}{4}\right\}}\left(|\nabla \pi_{h}|+|\tau_{h}|\right) \cdot|\bm{u}|\dx  \text{d}s  \notag\\
	\lesssim&  \sum_{k=4}^{n} \sum_{\nu} r_{k}^{-1}\left\||\nabla \pi_{h}|+|\tau_{h}|\right\|_{L^{\frac{3}{2}}\left(Q_{k}(R) \cap\left\{\left|x^{\prime}-x_{\nu}^{\prime}\right|<\frac{R-\rho}{4}\right\}\right)}\|u\|_{L^{3}\left(Q_{k}(R) \cap\left\{\left|x^{\prime}-x_{\nu}^{\prime}\right|<\frac{R-\rho}{4}\right\}\right)}\notag\\
	\lesssim& (R-\rho)^{\frac{4}{3}} \sum_{k=4}^{n} \sum_{\nu} r_{k}^{-\frac{1}{3}}\left\||\nabla \pi_{h}|+|\tau_{h}|\right\|_{L^{\frac{3}{2}}\left(-r_{k}^{2}, 0 ; L^{\infty}\left(U_{k}(R) \cap\left\{\left|x^{\prime}-x_{\nu}^{\prime}\right|<\frac{R-\rho}{4}\right\}\right)\right)}\notag\\
	&\quad\times\|u\|_{L^{3}\left(Q_{k}(R) \cap\left\{\left|x^{\prime}-x_{\nu}^{\prime}\right|<\frac{R-\rho}{4}\right\}\right)} \notag\\
	\lesssim& (R-\rho)^{2} \sum_{k=4}^{n} \sum_{\nu} r_{k}^{-\frac{1}{3}}\left\||\nabla \pi_{h}|+|\tau_{h}|\right\|_{L^{\frac{3}{2}}\left(-r_{k}^{2}, 0; L^{\infty}\left(U_{k}(R) \cap\left\{x^{\prime}-x_{\nu}^{\prime} \mid<\frac{R-\rho}{4}\right\}\right)\right)}^{\frac{3}{2}} \notag\\
	&+  \sum_{k=2}^{n} r_{k}^{-\frac{1}{3}}\|u\|_{L^{3}\left(Q_{k}(R)\right) }^{3}.
	\end{align}
	For any
	\begin{align*}
	x^{*} \in U_{k}(R) \cap\left\{x=\left(x^{\prime}, x_{3}\right) ;\left|x^{\prime}-x_{\nu}^{\prime}\right|<\frac{R-\rho}{4}\right\},
	\end{align*}
	we have
	\begin{align*}
	x^{*} \in B\left(x^* ; \frac{R-\rho}{4}\right) \subset U_{1}(R) \cap\left\{x=(x^{\prime},x_{3});\left|x^{\prime}-x_{\nu}^{\prime}\right|<\frac{R-\rho}{2}\right\},
	\end{align*}
	due to $k \geq 4$ and $|R-\rho| \leq \frac{1}{2} .$
	Since $\pi_{h},\tau_{h}$ are harmonic in $U_{1}(R),$ using the mean value property, we have
	\begin{align*}
	\left|\nabla \pi_{h}\right|\left(x^{*}\right)+|\tau_{h}|\left(x^{*}\right)  \lesssim & \frac{1}{|R-\rho|^{4}} \int_{B\left(x^* ; \frac{R-\rho}{4}\right)}|\pi_{h}|+|\tau_{h}|\dx   \\
	\lesssim& \frac{1}{(R-\rho)^{3}}\left\||\pi_{h}|+|\tau_{h}|\right\|_{L^ \frac{3}{2}\left(U_{1}(R) \cap\left\{\left|x^{\prime}-x_{\nu}^{\prime}\right|<\frac{R-\rho}{2}\right\}\right)},
	\end{align*}
	which implies
	\begin{align}\label{5.7}
	&\left\||\nabla \pi_{h}|+|\tau_{h}|\right\|_{L^{\frac{3}{2}}\left(-r_{k}^{2}, 0 ; L^{\infty}\left(U_{k}(R) \cap\left\{\left|x^{\prime}-x_{\nu}^{\prime} \right|<\frac{R-\rho}{4}\right\}\right)\right.}^{\frac{3}{2}} \notag\\
	=& \int_{-r_{k}^{2}}^{0}\left\||\nabla \pi_{h}|+|\tau_{h}|\right\|_{L^{\infty}\left(U_{k}(R) \cap\left\{\left|x^{\prime}-x_{\nu}^{\prime}\right|<\frac{R-\rho}{4}\right\}\right)}^{\frac{3}{2}} \text{d}s   \notag\\
	\lesssim& \frac{1}{(R-\rho)^{\frac{9}{2}}} \int_{-r_{k}^{2}}^{0} \int_{U_{1}(R)}\left(|\pi_{h}|+|\tau_{h}|\right)^{\frac{3}{2}}\chi_{B^{\prime}(x_{\nu}^{\prime};\frac{R-\rho}{2})}\dx  \text{d}s  .
	\end{align}
	Hence, combining \eqref{5.6} and \eqref{5.7} and taking sum over $\nu$ , we have
	\begin{align}\label{3.30}
     J_{2}+H_{2}\lesssim& \frac{1}{(R-\rho)^{\frac{5}{2}}} \sum_{k=2}^{n} r_{k}^{-\frac{1}{3}}\left(\left\||\pi_{h}|+|\tau_{h}|\right\|_{L^{\frac{3}{2}}\left(-r_{k}^{2}, 0 ; L^{\frac{3}{2}}\left(U_{1}(R)\right)\right)}^{\frac{3}{2}}+\|u\|_{L^{3}\left(-r_{k}^{2}, 0 ; L^{3}\left(U_{k}(R)\right)\right)}^{3}\right) \notag\\
     \lesssim& \frac{1}{(R-\rho)^{\frac{5}{2}}} \sum_{k=2}^{n} r_{k}^{\frac{1}{6}}\|u\|_{L^{4}\left(-r_{k}^{2}, 0 ; L^{3}\left(\R^3\right)\right)}^{3}\notag\\
     \lesssim& \frac{1}{(R-\rho)^{\frac{5}{2}}}\E^{\frac{3}{2}}.
	\end{align}
Summing up \eqref{3.25}, \eqref{3.26} and \eqref{3.30} , we have \eqref{H}.
\end{proof}

\begin{lem}\label{lem5}{\sl Under the assumption of Lemma \ref{lem1}, we have
	\begin{align}\label{K}
	K\leq C\frac{\E^{\frac{1}{2}}}{R-\rho} \sum_{i=0}^{n} r_{i}^{\frac{1}{2}}\left(r_{i}^{-1} E_{i}(R)\right).
	\end{align}}
\end{lem}
\begin{proof}
The proof is rather similar as Lemma 3.3 in \cite{Wang2020}. We omit the detail here.
\end{proof}
\begin{lem}\label{lem6} {\sl Under the assumption of Lemma \ref{lem1}, we have
\begin{align}\label{p1}
&\int_{-1}^{t} \int_{U_{0}(R)} \pi \bm{u} \cdot \nabla\left(\Phi_{n} \eta \psi\right)\dx  \normalfont{\text{d}s}  \notag\\
\leq& C\sum_{i=0}^{n} \B_{i} \left(r_{i}^{-1} E_{i}(R)\right)+C\frac{\E^{\frac{1}{2}}}{R-\rho}  \sum_{i=0}^{n} r_{i}^{\frac{1}{2}}\left(r_{i}^{-1} E_{i}(R)\right)+\frac{C}{(R-\rho)^{\frac{5}{2}}} \E^{\frac{3}{2}}.
\end{align}}
\end{lem}
\begin{proof}
Combining the estimates \eqref{pressure}, \eqref{J}, \eqref{J_{1}}, \eqref{H} and \eqref{K}, we have \eqref{p1}.
\end{proof}
\subsection{The proof of Theorem \ref{thm2}}\label{section3.3}
On the basis of the estimates of the nonlinear term and the pressure in subsection 3.1--3.2, we are in position to give the detail proof of  Theorem \ref{thm2}.
\begin{proof}
Gathering \eqref{key} and the estimates in Lemma \ref{lem1}, \ref{lem2} and  \ref{lem6}, we have
\begin{align*}
r_{n}^{-1} E_{n}(\rho)
\leq& C\sum_{i=0}^{n} \B_{i}  \left(r_{i}^{-1} E_{i}(R)\right)+ C\frac{ \E^{\frac{1}{2}}}{R-\rho} \sum_{i=0}^{n} r_{i}^{\frac{1}{2}}\left(r_{i}^{-1} E_{i}(R)\right)+C\frac{1+\E^{\frac{3}{2}}}{(R-\rho)^{\frac{5}{2}}}\\
\leq& C\sum_{i=0}^{n} \B_{i}  \left(r_{i}^{-1} E_{i}(R)\right)+C\frac{\E^{\frac{3}{4}}}{R-\rho}\sum_{i=0}^{n}r_{i}^{-\frac{5}{8}}E_{i}(R)^{\frac{3}{4}}r_{i}^{\frac{1}{8}}+C\frac{1+\E^{\frac{3}{2}}}{(R-\rho)^{\frac{5}{2}}}\\
\leq &  C\sum_{i=0}^{n} \B_{i}  \left(r_{i}^{-1} E_{i}(R)\right)+C\frac{\E^{\frac{3}{4}}}{R-\rho}\left(\sum_{i=0}^{n}r_{i}^{-\frac{5}{6}}E_{i}(R)\right)^{\frac{3}{4}}\left(\sum_{i=0}^{n}r_{i}^{\frac{1}{2}}\right)^{\frac{1}{4}}\\
&+C\frac{1+\E^{\frac{3}{2}}}{(R-\rho)^{\frac{5}{2}}}\\
\leq&C_{0}\sum_{i=0}^{n} \left(\B_{i}+r_{i}^{\frac{1}{6}}\right)  \left(r_{i}^{-1} E_{i}(R)\right)+C_{0}\frac{1+\E^{3}}{(R-\rho)^{4}}.
\end{align*}
In view of \eqref{B}, there exists a sufficient large number $n_{0}\geq 1$ such that
\begin{equation}\label{small}
C_{0}\sum_{i=n_{0}}^{\infty} \left(\B_{i}+r_{i}^{\frac{1}{6}}\right)\leq \frac{1}{2}.
\end{equation}
Then for $n\geq n_{0} $ we have
\begin{align}\label{iteration1}
r_{n}^{-1} E_{n}(\rho) \leq& C_{0}\sum_{i=n_{0}}^{n} \left(\B_{i}+r_{i}^{\frac{1}{6}}\right)  \left(r_{i}^{-1} E_{i}(R)\right)\nonumber\\
&\quad+C_{0}\sum_{i=0}^{n_{0}-1} \left(\B_{i}+r_{i}^{\frac{1}{6}}\right)  \left(r_{i}^{-1} E_{i}(R)\right)+C_{0}\frac{1+\E^{3}}{(R-\rho)^{4}}\notag\\
\leq& C_{0}\sum_{i=n_{0}}^{n} \left(\B_{i}+r_{i}^{\frac{1}{6}}\right)  \left(r_{i}^{-1} E_{i}(R)\right)+\frac{A_{0}}{(R-\rho)^{4}},
\end{align}
where the constant  $A_{0}=C_{1}\cdot 2^{n_{0}}\left(1+\E^2\cdot\|\partial_{3}\bm{u}\|_{L^{p_{0}, 1}\left(-1,0;L^{q_{0}}(B(2))\right)}+\E^{3}\right)$.

As the iteration argument in \cite[V. Lemma 3.1]{Giaquinta1983}, we introduce the sequence $\{\rho_{k}\}_{k=0}^{+\infty}$ which satisfies
\begin{equation*}
\rho_{0}=\frac{1}{2},\quad \rho_{k+1}-\rho_{k}=\frac{1-\theta}{2}\theta^{k},
\end{equation*}
with  $\frac{1}{2}<\theta^4<1$ and $\lim_{k\rightarrow\infty}\rho_{k}=1$. For $n_{0}\leq j\leq n$ and $k\geq 1$, we have
\begin{align}\label{iteration2}
r_{j}^{-1} E_{j}(\rho_{k})\leq& C_{0}\sum_{i=n_{0}}^{j} \left(\B_{i}+r_{i}^{\frac{1}{6}}\right)  \left(r_{i}^{-1} E_{i}(\rho_{k+1})\right)+A_{0}\cdot \frac{16}{(1-\theta)^4}\theta^{-4k}\notag\\
\leq&C_{0}\sum_{i=n_{0}}^{n} \left(\B_{i}+r_{i}^{\frac{1}{6}}\right)  \left(r_{i}^{-1} E_{i}(\rho_{k+1})\right)+A_{0}\cdot \frac{16}{(1-\theta)^4}\theta^{-4k}.
\end{align}
By using  iteration argument from \eqref{iteration1}, and then applying \eqref{small} and \eqref{iteration2}, we obtain that for $n\geq n_{0}$,
\begin{align*}
r_{n}^{-1} E_{n}(\rho_{0}) \leq&C_{0}\sum_{j=n_{0}}^{n} \left(\B_{j}+r_{j}^{\frac{1}{6}}\right)  \left(r_{j}^{-1} E_{j}(\rho_{1})\right)+A_{0}\cdot \frac{16}{(1-\theta)^4}\\
\leq&\frac{1}{2} C_{0}\sum_{j=n_{0}}^{n} \left(\B_{j}+r_{j}^{\frac{1}{6}}\right)  \left(r_{j}^{-1} E_{j}(\rho_{2})\right)+A_{0}\cdot \frac{16}{(1-\theta)^4}\left(1+\frac{1}{2\theta^4}\right)\\
\leq& \frac{1}{2^{k-1}}C_{0}\sum_{j=n_{0}}^{n} \left(\B_{j}+r_{j}^{\frac{1}{6}}\right)  \left(r_{j}^{-1} E_{j}(\rho_{k})\right)+A_{0}\cdot \frac{16}{(1-\theta)^4}\sum_{j=0}^{k-1} \left(\frac{1}{2\theta^4}\right)^j\\
\leq&\frac{1}{2^{k-1}}C_{0}\sum_{j=n_{0}}^{n} \left(\B_{j}+r_{j}^{\frac{1}{6}}\right)  \left(r_{j}^{-1} E_{j}(1)\right)+A_{0}\cdot \frac{16}{(1-\theta)^4}\cdot \frac{2\theta^4}{2\theta^4-1}.
\end{align*}
Let $k\rightarrow\infty$, we obtain that for $n\geq n_{0}$,
\begin{equation}\label{3.38}
r_{n}^{-1} E_{n}(\rho_{0}) \leq A_{0}\cdot \frac{16}{(1-\theta)^4}\cdot \frac{2\theta^4}{2\theta^4-1}.
\end{equation}
For $0<r\leq r_{n_{0}}$, there exists $n_{1}\geq n_{0}$, such that
\begin{equation*}
r_{n_{1}+1}<r\leq r_{n_{1}},
\end{equation*}
which together with \eqref{3.38} ensures that
\begin{align*}
r^{-2}\|u\|_{L^{3}\left(B(r) \times\left(-r^{2}, 0\right)\right)}^{3} \leq C r^{-\frac{3}{2}}\|u\|_{L^{4}\left(-r^{2}, 0 ; L^{3}(B(r))\right)}^{3} \leq C\left(r_{n_{1}}^{-1} E_{n_{1}}(\frac{1}{2})\right)^{\frac{3}{2}} \leq C.
\end{align*}
The proof is completed.
\end{proof}

\appendix

\section{}
\begin{lem}\label{A1}
{\sl 	Let $0<r \leq R<+\infty $ and $h: B^{\prime}(2 R) \times(-r, r) \rightarrow \R$ be harmonic. Then for all $0<\rho \leq \frac{r}{4}$ and $1 \leq \ell \leq q <+\infty$, we get
	\begin{align}
	\|\nabla^m h\|_{L^{q}\left(B^{\prime}(R) \times(-\rho, \rho)\right)}^{q} \leq C \rho r^{2-mq- \frac{3q}{\ell}}\|h\|_{L^{\ell}\left(B^{\prime}(2 R) \times(-r, r)\right)}^{q},\ m\in\N,
	\end{align}
	where $C$ stands for a positive constant depending only on $q,m$ and $\ell$.}
\end{lem}
\begin{proof}
	The proof is similar as Lemma A.2 in \cite{Chae2019}. We choose a finite family of points $\left\{x_{\nu}^{\prime}\right\}$ in $B^{\prime}(R)$ such that $\left\{B^{\prime}\left(x_{\nu}^{\prime}, r / 4\right)\right\}$ is a covering of $\overline{B^{\prime}(R)},$ and it holds
	\begin{align}
	\sum_{\nu} \chi_{B^{\prime}\left(x_{\nu}^{\prime}, r\right)} \leq C.
	\end{align}
 Setting $x_{\nu}=\left(x_{\nu}^{\prime}, 0\right),$ we see that
	$$B^{\prime}\left(x_{\nu}^{\prime}, r / 4\right) \times(-r / 4, r / 4) \subset B\left(x_{\nu}, r / 2\right) .$$
	With this notation, we have
	\begin{align}\label{4.9}
	\|\nabla^m h\|_{L^{q}\left(B^{\prime}(R) \times\left(-\rho, \rho\right)\right)}^{q} & \leq \sum_{\nu}\|\nabla^m h \|_{L^{q}\left(B^{\prime}\left(x_{\nu}^{\prime}, r / 4\right) \times\left(-\rho, \rho\right)\right)}^{q} \notag\\
	& \leq C r^{2} \rho \sum_{\nu}\|\nabla^m h\|_{L^{\infty}\left(B^{\prime}\left(x_{\nu}^{\prime}, r / 4\right) \times(-r / 4, r / 4)\right)}^{q}\notag\\
	&\leq C r^{2} \rho \sum_{\nu}\|\nabla^m h\|_{L^{\infty}\left(B\left(x_{\nu}, r / 2\right)\right)}^{q}.
	\end{align}
	Since $h$ is harmonic, using the mean value property and
	taking the sum over $\nu$, we obtain
	\begin{align}\label{4.10}
	\sum_{\nu}\|\nabla^m h\|_{L^{\infty}\left(B\left(x_{\nu}, r / 2\right)\right)}^{q} \leq C r^{-mq-\frac{3q}{\ell}}\|h\|_{L^{\ell}\left(B^{\prime}(2 R) \times(-r, r)\right)}^{q}.
	\end{align}
	Combining \eqref{4.9} and \eqref{4.10}, we get
	\begin{align}
	\|\nabla^m h\|_{L^{q}\left(B^{\prime}(R) \times\left(-\rho, \rho\right)\right)}^{q} \leq C \rho r^{2-mq- \frac{3q}{\ell}}\|h\|_{L^{\ell}\left(B^{\prime}(2 R) \times(-r, r)\right)}^{q}.
	\end{align}
\end{proof}
\begin{cor}\label{A2}{\sl For $1<\ell \leq q<+\infty$ and $k\geq j+4$, we have
	\begin{equation}\label{pi_{0,j}}
	\|\pi_{0,j}\|_{L^{q}(U_{k}(R))}\leq C r_{k}^{\frac{1}{q}} r_{j}^{\frac{2}{q}-\frac{3}{\ell}} \|\bm{u}\|_{L^{2\ell}(U_{j}(R))}^2 ,
	\end{equation}
	\begin{equation} \label{tau_{0,j}}
	\|\tau_{0,j}\|_{L^{q}(U_{k}(R))}\leq C r_{k}^{\frac{1}{q}} r_{j}^{\frac{2}{q}-\frac{3}{\ell}-1} \|\bm{u}\|_{L^{2\ell}(U_{j}(R))}^2 ,
	\end{equation}
and
\begin{equation}\label{tau'_{0,j}}
	\|\tau_{0,j}\|_{L^{q}(U_{k}(R))}\leq C r_{k}^{\frac{1}{q}} r_{j}^{\frac{2}{q}-\frac{3}{\ell}} \|\partial_{3}\bm{u}\otimes\bm{u}\|_{L^{\ell}(U_{j}(R))} .
\end{equation}}
\end{cor}
\begin{proof}
	We recall that $\pi_{0,j}=\J\left(\bm{u}\otimes\bm{u}\cdot\phi_{j}\right)$ and $\tau_{0,j}=\J\left(\partial_{3}\left(\bm{u}\otimes\bm{u}\right)\cdot\phi_{j}\right)$. Hence
	\begin{equation}
	\tau_{0,j}=\partial_{3}\pi_{0,j}-\tau_{h,j}
	\end{equation}
	with $\tau_{h,j}=\J\left(\bm{u}\otimes\bm{u}\cdot\partial_{3}\phi_{j}\right)$. From the definition of $\phi_{j}$ in \eqref{phi} , it follows that the function $\pi_{0,j},\tau_{0,j},\tau_{h,j}$ are harmonic in $\R^2\times (-r_{j+2},r_{j+2})\times (-r_{j+2}^2,0)$. Applying \eqref{ani} and Lemma \ref{A1} with $r=r_{j+2},\rho=r_{k}$, we have \eqref{pi_{0,j}},  \eqref{tau_{0,j}} and \eqref{tau'_{0,j}} directly.
\end{proof}
\begin{lem}[Poincar\'{e}'s inequality]\label{A3}
{\sl	Set $\overline{(h)}_{j}(x^{\prime})=\frac{1}{2r_{j}}\int_{-r_{j}}^{r_{j}}h(x^{\prime},\omega)~d\omega$. For $k\geq j$, it holds
	\begin{align}\label{8}
	\|h-\overline{(h)}_{j}\|_{L_{v}^{q_{v}}L_{h}^{q_{h}}\left(B^{\prime}(R) \times\left(-r_{k}, r_{k}\right)\right)}\leq C r_{k}^{\frac{1}{q_{v}}}r_{j}^{1-\frac{1}{\ell}} \|\partial_{3} h\|_{L_{v}^{\ell}L_{h}^{q_{h}}\left(B^{\prime}(R) \times\left(-r_{j}, r_{j}\right)\right)},
	\end{align}
	and
	\begin{align}\label{5.9}
	\|\overline{(h)}_{j}\|_{L_{v}^{q_{v}}L_{h}^{q_{h}}\left(B^{\prime}(R) \times\left(-r_{k}, r_{k}\right)\right)}\leq Cr_{k}^{\frac{1}{q_{v}}}r_{j}^{-\frac{1}{\ell}}\|h\|_{L_{v}^{\ell}L_{h}^{q_{h}}\left(B^{\prime}(R) \times\left(-r_{j}, r_{j}\right)\right)}.
	\end{align}}
\end{lem}
\begin{proof} For \eqref{8}, we see that
	\begin{align*}
	|h(x^{\prime},x_{3})-\overline{(h)}_{j}|
\leq&\frac{1}{2r_{j}}\int_{-r_{j}}^{r_{j}}
\left|h(x^{\prime},x_{3})-h(x^{\prime},\omega)\right|~\text{d}\omega\\
	\leq&\frac{1}{2r_{j}}\int_{-r_{j}}^{r_{j}}\left|\int_{\omega}^{x_{3}} \partial_{3}h(x^{\prime},\xi)~\text{d}\xi\right|~\text{d}\omega\\
	\leq&\int_{-r_{j}}^{r_{j}} \left|\partial_{3}h(x^{\prime},\xi)\right|~\text{d}\xi.
	\end{align*}	
	Applying H\"{o}lder's  and Minkowski inequality, we have
	\begin{align*}
	\|h-\overline{(h)}_{j}\|_{L_{v}^{q_{v}}L_{h}^{q_{h}}\left(B^{\prime}(R) \times\left(-r_{k}, r_{k}\right)\right)}\lesssim& r_{k}^{\frac{1}{q_{v}}}\cdot \int_{-r_{j}}^{r_{j}}\|\partial_{3}h(\cdot,\xi)\|_{L_{h}^{q_{h}}(B^{\prime}(R))}~\text{d}\xi\\
	\lesssim&r_{k}^{\frac{1}{q_{v}}}r_{j}^{1-\frac{1}{\ell}} \|\partial_{3} h\|_{L_{v}^{\ell}L_{h}^{q_{h}}\left(B^{\prime}(R) \times\left(-r_{j}, r_{j}\right)\right)}.
	\end{align*}
As to \eqref{5.9}, we apply H\"{o}lder's  and Minkowski inequalities again to get
	\begin{align*}
	\|\overline{(h)}_{j}\|_{L_{v}^{q_{v}}L_{h}^{q_{h}}\left(B^{\prime}(R) \times\left(-r_{k}, r_{k}\right)\right)} \lesssim& r_{k}^{\frac{1}{q_{v}}}\cdot \frac{1}{2r_{j}}\int_{-r_{j}}^{r_{j}}\|h(\cdot,\omega)\|_{L^q(B^{\prime}(R))}~\text{d}\omega\\
	\lesssim&r_{k}^{\frac{1}{q_{v}}}r_{j}^{-\frac{1}{\ell}}\|h\|_{L_{v}^{\ell}L_{h}^{q_{h}}\left(B^{\prime}(R) \times\left(-r_{j}, r_{j}\right)\right)}.
	\end{align*}
The proof of this lemma is completed.
\end{proof}

\begin{lem}[Lemma A.2 in \cite{Wang2020}]\label{A4}{\sl
Let $0< p,\sigma<+\infty .$ Then for any $h(s) \in L^{p,\sigma}(\R),$ there exists a sequence $\left\{c_{n}\right\}_{n \in \mathbb{Z}} \in \ell^{\sigma}$ and sequence of functions $\left\{h_{n}\right\}_{n \in \mathbb{Z}}$ with each $h_{n}$ bounded by $2^{\frac{n}{p}}$ and supported on a set of measure $2^{-n}$. Moreover,
\begin{align*}
h=\sum_{n \in \mathbb{Z}} c_{n} h_{n}
\end{align*}
and
\begin{align*}
c(p, \sigma)\left\|\left\{c_{n}\right\}\right\|_{\ell^{\sigma}} \leq\|h\|_{L^{p, \sigma}} \leq C(p, \sigma)\left\|\left\{c_{n}\right\}\right\|_{\ell^{\sigma}},
\end{align*}
where the constant $c(p, \sigma)$ and $C(p, \sigma)$ only depend on $p, \sigma$.}
\end{lem}

\begin{lem}\label{A5}
{\sl For any
\begin{align*}
1 \leq p^*<p=\frac{2q}{2q-3}, \quad \frac{3}{2}<q<+\infty,
\end{align*}
we have
\begin{align*}
\sum_{k=0}^{+\infty} r_{k}^{2-\frac{2}{p^*}-\frac{3}{q}}\left(\int_{-r_{k}^{2}}^{0}\left\|\partial_{3}\bm{u}(\cdot, s)\right\|_{L^q(B(2))}^{p^*} ~ds\right)^{\frac{1}{p^*}} \leq C\left\|\partial_{3}\bm{u} \right\|_{L^{p, 1}\left(-1,0;L^{q}(B(2))\right) }.
\end{align*}}
\end{lem}
\begin{proof}
Let $h(s)=\left\|\partial_{3} \bm{u}(\cdot, s)\right\|_{L^{q}(B(2))},-1< s<0$. By Lemma \ref{A4}, we know that
\begin{align*}
h=\sum_{j=0}^{+\infty} c_{j} h_{j}, \quad\|h\|_{L^{p, 1}} \approx \sum_{j=0}^{+\infty}\left|c_{j}\right|,
\end{align*}
where
\begin{align*}
\left|h_{j}\right| \leq 2^{\frac{j}{p}}, \quad\left|D_{j}=\operatorname{supp} h_{j}\right| \leq 2^{-j}.	\end{align*}
We can restrict $j\geq0$ here, since $s \in (-1,0)$ and the construction of the function $h_{j}$ in Lemma \ref{A4} by standard atomic decomposition. Then we have
\begin{align*}
&\sum_{k=0}^{+\infty} r_{k}^{2-\frac{2}{p^*}-\frac{3}{q}}\left(\int_{-r_{k}^{2}}^{0}\left\|\partial_{3}\bm{u}(\cdot, s)\right\|_{L^q(B(2))}^{p^*} ~ds\right)^{\frac{1}{p^*}}\\
=&\sum_{k=0}^{+\infty} r_{k}^{2-\frac{2}{p^*}-\frac{3}{q}}\left(\int_{I_{k}}|h|^{p^*} ~ds\right)^{\frac{1}{p^*}} \\
\leq& \sum_{k=0}^{+\infty} r_{k}^{2-\frac{2}{p^*}-\frac{3}{q}} \sum_{j=0}^{+\infty}\left|c_{j}\right| 2^{\frac{j}{p}}\left|D_{j} \cap I_{k}\right|^{\frac{1}{p^*}} \\
\leq& \sum_{j=0}^{+\infty}\left|c_{j}\right| 2^{\frac{j}{p}} \sum_{k=0}^{+\infty} r_{k}^{2-\frac{2}{p^*}-\frac{3}{q}}\left|D_{j} \cap I_{k}\right|^{\frac{1}{p^*}}\\
\leq& \sum_{j=0}^{+\infty}\left|c_{j}\right| 2^{\frac{j}{p}} \sum_{k\leq \frac{j}{2}} r_{k}^{2-\frac{2}{p^*}-\frac{3}{q}}\left|D_{j} \cap I_{k}\right|^{\frac{1}{p^*}}+\sum_{j=0}^{+\infty}\left|c_{j}\right| 2^{\frac{j}{p}} \sum_{k>\frac{j}{2}} r_{k}^{2-\frac{2}{p^*}-\frac{3}{q}}\left|D_{j} \cap I_{k}\right|^{\frac{1}{p^*}},
\end{align*}
where $I_{k}=\left(-r_{k}^{2}, 0\right) .$ On the other hand, we notice that for any $k \leq \frac{j}{2}$
\begin{equation}\label{A11}
r_{k}^{2-\frac{2}{p^*}-\frac{3}{q}}\left|D_{j} \cap I_{k}\right|^{\frac{1}{p^*}} \leq C 2^{-\frac{j}{p^*}} 2^{-k\left(2-\frac{2}{p^*}-\frac{3}{q}\right)},
\end{equation}
and for any $\frac{j}{2} < k<+\infty$
\begin{equation}
r_{k}^{2-\frac{2}{p^*}-\frac{3}{q}}\left|D_{j} \cap I_{k}\right|^{\frac{1}{p^*}} \leq C 2^{-\frac{2 k}{p^*}} 2^{-k\left(2-\frac{2}{p^*}-\frac{3}{q}\right)}=C 2^{-k\left(2-\frac{3}{q}\right)}.
\end{equation}
Accordingly,
\begin{align*}
&\sum_{k=0}^{+\infty} r_{k}^{2-\frac{2}{p^*}-\frac{3}{q}}\left(\int_{-r_{k}^{2}}^{0}\left\|\partial_{3}\bm{u}(\cdot, s)\right\|_{L^q(B(2))}^{p^*} ~ds\right)^{\frac{1}{p^*}} \\
\lesssim& \sum_{j=0}^{+\infty}\left|c_{j}\right| 2^{\frac{j}{p}} \sum_{k\leq \frac{j}{2}} 2^{-\frac{j}{p^*}} 2^{-k\left(2-\frac{2}{p^*}-\frac{3}{q}\right)}+\sum_{j=0}^{+\infty}\left|c_{j}\right| 2^{\frac{j}{p}} \sum_{k>\frac{j}{2}} 2^{-k\left(2-\frac{3}{q}\right)}\\
\lesssim&  \sum_{j=0}^{+\infty}\left| c_{j} \right| \approx \left\|\partial_{3}\bm{u}\right\|_{L^{p,1}\left(-1,0;L^{q}(B(2))\right) }.
\end{align*}
This completes the proof of this lemma.
\end{proof}
\section*{Acknowledgments}
H. Chen was supported by Zhejiang Province Science Fund for Youths [LQ19A010002]. C. Qian  is supported by the Natural Science Foundation of Zhejiang Province [LY20A010017].

\end{document}